\documentclass[12pt,a4paper,oneside]{amsart}

\usepackage{amsthm}
\usepackage{amsmath}
\usepackage{amssymb}
\usepackage{amscd}
\usepackage[utf8]{inputenc}
\usepackage{typearea}
\usepackage{eufrak}
\usepackage{yfonts}
\usepackage{textcomp}
\usepackage{mathrsfs}
\usepackage{hyperref}
\usepackage[draft]{fixme} 
\usepackage{pdfsync}
\usepackage[active]{srcltx}
\usepackage{xypic}
\usepackage{verbatim}

\renewcommand{\bar}[1]{\overline{#1}}

\newcommand{\ZZ}{\mathbb{Z}}
\newcommand{\NN}{\mathbb{N}}

\newcommand{\Egr}[1][\Asf]{E_{#1}}

\newcommand{\g}{\mathbf{g}}

\newcommand{\G}{\mathcal{G}}
\newcommand{\SAB}{\mathcal{S}_{\Asf,\Bsf}}

\newcommand{\HG}{\mathcal{H}}

\newcommand{\RG}{\mathcal{R}}

\newcommand{\un}{\mathbf{u}}

\newcommand{\Asf}{\mathsf{A}}
\newcommand{\Bsf}{\mathsf{B}}

\newcommand{\Ide}{\mathsf{I}}

\newcommand{\Ind}{\mathbf{1}}
\newcommand{\coker}{\operatorname{coker}}

\newtheorem{lemma}{Lemma}[section]
\newtheorem{corollary}[lemma]{Corollary}
\newtheorem{theorem}[lemma]{Theorem}
\newtheorem{proposition}[lemma]{Proposition}

\theoremstyle{definition}
\newtheorem{definition}[lemma]{Definition}

\newtheorem{example}[lemma]{Example}

\newtheorem{remark}[lemma]{Remark}

\newtheorem*{theorem*}{Theorem}

\title[The homology of the Katsura-Exel-Pardo groupoid]{The homology of the Katsura-Exel-Pardo groupoid}
 
 \author{Eduard Ortega}
 \address{Department of Mathematical Sciences\\NTNU\\NO-7491 Trondheim\\Norway}
 \email{eduard.ortega@ntnu.no}
 
\date{\today}


\begin{document}
\maketitle
 
 \begin{abstract}
We compute the homology of the groupoid associated to the Katsura algebras, and show that they capture the $K$-theory of the $C^*$-algebras in the sense of the (HK) conjecture posted by Matui.  Moreover, we show that several classifiable simple $C^*$-algebras are groupoid $C^*$-algebras of this class.
 \end{abstract}
 
\section*{Introduction} 

In \cite{Kat} Katsura defined a nice class of $C^*$-algebras that exhausts all the Kirchberg algebras in the UCT class. The construction of these $C^*$-algebras has two layers: the first is the graph skeleton, that gives to the $C^*$-algebra most of the  desired structural properties, and the second layer that consists of partial unitaries associated to every vertex, which provide the necessary richness in $K$-theory. These two layers are given by two equal size square matrices $\Asf$ and $\Bsf$.  

Later in \cite{EP} Exel and Pardo, while studying the $C^*$-algebras associated to self-similar graphs,  realized that the $C^*$-algebras constructed by Katsura  were prominent examples of this construction. The advantage of Exel and Pardo approach is that  they described these algebras as groupoids $C^*$-algebras of combinatorial origin, and managed to give beautiful characterizations of the most fundamental properties of groupoids. In particular, for the Katsura algebras they construct an amenable groupoid $\G_{\Asf,\Bsf}$ such that $C^*(\G_{\Asf,\Bsf})$ is the desired $C^*$-algebra and give conditions, in most of the cases equivalent conditions, in terms of the matrices $\Asf$ and $\Bsf$ for Hausdorffness, effectiveness and minimality of the groupoid. Because Katsura found the $C^*$-algebra, but Exel and Pardo gave the description as a grupoid $C^*$-algebras, we choose to call   $\G_{\Asf,\Bsf}$ the \emph{Katsura-Exel-Pardo groupoid}.

As mentioned above, Katsura computed the $K$-theory of $C^*(\G_{\Asf,\Bsf})$ in terms of the matrices $\Asf$, $\Bsf$, that is 
$$K_0(C^*(\G_{\Asf,\Bsf}))\cong \coker(\Ide-\Asf)\oplus \ker (\Ide-\Bsf)\qquad\text{and} \qquad K_1(C^*(\G_{\Asf,\Bsf}))\cong \coker(\Ide-\Bsf)\oplus \ker (\Ide-\Asf)\,,$$
and showed that given any two countably generated abelian groups $G_0$ and $G_1$ there exist matrices $\Asf$ and $\Bsf$ such that $K_0(C^*(\G_{\Asf,\Bsf}))\cong G_0$ and $K_1(C^*(\G_{\Asf,\Bsf}))\cong G_1$.

In \cite{Mat2,Mat3} Matui started an exhaustive study of étale groupoids with totally disconnected unit space, and showed how their homology reflects dynamical properties of their topological full groups. He later conjectured in \cite{Mat4} that the homology groups of a minimal effective, étale  groupoid totally captures the $K$-theory of their associated reduced groupoid $C^*$-algebra, and called it the (HK) conjecture. He verified that the (HK) conjecture is true for important classes of groupoids, like the transformation groupoids of Cantor minimal systems, Cuntz-Krieger groupoids and products of Cuntz-Krieger groupoids.

In the present paper, we verify the conjecture for the class of Katsura-Exel-Pardo groupoids, that is, we compute all the homology groups of the groupoid $\G_{\Asf,\Bsf}$, and show that they sum up to the $K$-theory of the $C^*$-algebra. Furthermore, we see that homology groups provide a refinement of the $K$-theory allowing us to define invariants for the Kakutani equivalence class of the groupoid $\G_{\Asf,\Bsf}$ that could not be found just looking at the $K$-theory of the associated $C^*$-algebras. 
It was proved by Matsumoto and Matui \cite[Corollary 3.8]{MM} that given two irreducible matrices $\Asf$ and $\Asf'$, the Cuntz-Krieger groupoids $\G_{\Asf,0}$ and $\G_{\Asf',0}$ are equivalent if and only if $\coker (\Ide-\Asf)\cong \coker (\Ide-\Asf')$ and $\det(\Ide-\Asf)=\det(\Ide-\Asf')$. It then looks natural to go for a classification result for the Katsura-Exel-Pardo groupoids. Then we obtain
the following Main Theorem of the present paper.

\begin{theorem*}
Let $N\in \NN\cup\{\infty\}$, and let $\Asf$ and $\Bsf$ be two $N\times N$ row-finite matrices with integer entries, and such that $\Asf_{i,j}\geq 0$ for all $i$  and $j$. Moreover, suppose that  $\Bsf_{i,j}=0$ if and only if $\Asf_{i,j}=0$. Then 
\begin{align*}
H_0(\G_{\Asf,\Bsf})& \cong \coker (\Ide-\Asf) & \qquad & H_1(\G_{\Asf,\Bsf}) \cong \ker(\Ide-\Asf)\oplus \coker (\Ide-\Bsf) \\
H_2(\G_{\Asf,\Bsf})& \cong \ker (\Ide-\Bsf)\,, &\qquad & H_{i}(\G_{\Asf,\Bsf}) =0\text{ for }i\geq 3\,.
\end{align*}
Therefore, $\G_{\Asf,\Bsf}$ satisfies the $(HK)$ conjecture.
\end{theorem*}

  In the Cuntz-Krieger case, the part of the invariant involving the determinant is contained in the first cohomology group of the groupoid, which is isomorphic to the Boyle-Handelman group \cite[Proposition 3.4]{MM}, while the cohomology of the Katsura-Exel-Pardo groupoid is much bigger and contains parts of the Boyle-Handelman group. So further study of this group is needed.  It is then the aim of this paper to set the first step in a future classification of the groupoids $\G_{\Asf,\Bsf}$ analyzing the combinatorial structure that they possess. 

Recently it has been a big interest in finding which classifiable $C^*$-algebras can be realized as étale groupoid $C^*$-algebras  (see \cite{Li,  LiRe,Put}). In order to do that one wants to construct groupoids  whose associated $C^*$-algebras exhaust the possible  Elliott invariants. Here is where étale groupoids satisfying the (HK) conjecture gain importance, since in general $K$-theory is an important part of this invariant.  

The paper is organized as follows. In section $1$ we give the preliminaries on étale groupoids and their homology. Here is where we state Lemma \ref{exact_seq}, that is the analog of the Pimsner-Voiculescu 6-terms exact sequence of $K$-theory, but for the homology of étale groupoids with a $\ZZ$-cocyle. This Lemma will be the crucial technical tool  for the computation of the homology of the Katsura-Exel-Pardo groupoid. In section $2$ we introduce the Katsura-Exel-Pardo groupoid, that is, a groupoid associated to a self-similar graph introduced by Exel and Pardo in \cite{EP} that realizes the $C^*$-algebra defined by Katsura \cite{Kat}. After a quick overview of the basics properties of this groupoid found in \cite[Section 18]{EP}, we move to the computation of the homology.  This is done in two steps: the first computes the homology of the kernel groupoid of the natural $\ZZ$-cocyle of $\G_{\Asf,\Bsf}$, denoted by $\HG_{\Asf,\Bsf}$. The second step is to use the long exact sequence found in  Lemma \ref{exact_seq} to compute the homology of $\G_{\Asf,\Bsf}$. This long exact sequence contains the homology groups of $\HG_{\Asf,\Bsf}$ with maps induced by the dual action of the $\ZZ$-cocyle. Thanks to the nice description of these maps given in Proposition \ref{trans_map} and the nature of the homology groups of $\HG_{\Asf,\Bsf}$ the homology groups of $\G_{\Asf,\Bsf}$ fits in short exact sequences and hence be computed. Finally in section $3$ we use Theorem \ref{HK_conj} to construct a variety of étale groupoids whose associated $C^*$-algebra fall in a classifiable class and with a prescribed $K$-theoretical invariant.

\section{Basics on groupoid homology.}

In this section we will recall the basic definitions and results on groupoid homology that one can find in \cite{Mat2}, and we will state the conjecture of study in this paper.

A \emph{groupoid} is a small category of isomorphisms, that is, a set $\G$ (the morphisms, or arrows in the category) equipped with a partially defined multiplication $(g_1, g_2) \mapsto g_1 \cdot g_2$ for a distinguished subset $\G^{(2)} \subseteq \G \times \G$, and everywhere defined involution $g \mapsto g^{-1}$ satisfying the following axioms:
\begin{enumerate}
	\item If $g_1g_2$ and $(g_1g_2)g_3$ are defined, then $g_2g_3$ and $g_1(g_2g_3)$ are defined and $(g_1g_2)g_3=g_1(g_2g_3)$,
	\item The products $gg^{-1}$ and $g^{-1}g$ are always defined. If $g_1g_2$ is defined, then $g_1=g_1g_2g_2^{-1}$ and $g_2=g_1^{-1}g_1g_2$.
\end{enumerate}

A \emph{topological groupoid} is a groupoid together with a topology on it such that the operations of multiplication and taking inverse are continuous. 

The elements of the form $gg^{-1}$ are called \emph{units}. We denote the set of units of a groupoid $\G$ by $\G^{(0)}$, and refer to this as the \emph{unit space}. We always think of the unit space as a topological space equipped with the relative topology from $\mathcal{G}$. The \emph{source} and \emph{range} maps are
$$s(g) := g^{-1}g\qquad\text{and}\qquad r(g) := gg^{-1}$$ 
for $g\in \G$.

An \emph{étale} groupoid is a topological groupoid where the range map (and necessarily the source map) is a local homeomorphism (as a map from $\mathcal{G}$ to $\mathcal{G}$). The unit space $\mathcal{G}^{(0)}$ of an étale groupoid is always an open subset of $\mathcal{G}$.

\begin{definition}
	Let $\mathcal{G}$ be an étale groupoid. A \emph{bisection} is an open subset $U\subseteq \G$ such that $s$ and $r$ are both injective when restricted to $U$. 
\end{definition}

Two units $x,y\in \G^{(0)}$ belong to the same \emph{$\G$-orbit} if there exists $g\in \G$ such that $s(g)=x$ and $r(g)=y$.  We denote by $\text{orb}_{\G}(x)$ the $\G$-orbit of $x$. When every $\G$-orbit is dense in $\G^{(0)}$, $\G$ is called \emph{minimal}. An open set $A$ is called \emph{$\G$-full} if for every $x\in \G^{(0)}$ one has $\text{orb}_{\G}(x)\cap A\neq \emptyset$.

For an open subset $A\subseteq \G^{(0)}$ we denote by $\G_{A}$ the subgroupoid $\{g\in \G \mid s(g),r(g)\in A \}$, called the \emph{restriction} of $\G$ to $A$. When $\mathcal{G}$ is étale, the restriction $\G_{A}$ is an open étale subgroupoid.

The \emph{isotropy group} of a unit $x\in \G^{(0)}$ is the group $\G_x^x := \{g\in \G \mid s(g)=r(g)=x\}$, and the \emph{isotropy bundle} is
\[\G' := \{g\in \G \mid s(g)=r(g)\} = \bigcup_{x \in \mathcal{G}^{(0)}} \G_x^x.\]
A groupoid $\G$ is said to be \emph{principal} if all isotropy groups are trivial, or equivalently, $\G' = \mathcal{G}^{(0)}$. We say that $\G$ is \emph{effective} if the interior of $\G'$ equals $\mathcal{G}^{(0)}$.

\begin{definition}
	We say that a groupoid whose unit space is totally disconnected is \emph{elementary} if it is compact and principal. A groupoid $\G$ is an $AF$ groupoid if there exists an ascending chain of open elementary subgroupoids $K_1,K_2,\ldots$ such that $\G=\bigcup_{i=1}^\infty K_i$.  
\end{definition}

 Let $\pi:X\to Y$ be a local homeomorphism  between two locally compact Hausdorff spaces, then given any $f\in C_c(X,\ZZ)$ we define 
$$\pi_*(f)(y):=\sum_{\pi(x)=y}f(x)\,.$$
It is not hard to show that $\pi_*(f)\in C_c(Y,\ZZ)$. 

Given an étale groupoid $\G$ and $n\in\NN$ we write $\G^{(n)}$ for the space of composable strings of $n$ elements in $\G$ with the product topology. For $i=0,\ldots,n$, we let $d_i:\G^{(n)}\to\G^{(n-1)}$ be a map defined by 
$$d_i(g_1,g_2,\ldots,g_n)=\left\lbrace \begin{array}{ll}
(g_2,g_3,\ldots,g_n) & \text{if }i=0\,, \\
(g_1,\ldots, g_{i}g_{i+1},\ldots, g_n) & \text{if }1\leq i\leq n-1\,, \\
(g_1,g_2,\ldots,g_{n-1})& \text{if }i=n\,.
\end{array}\right. $$
  
  Then we define the homomorphism $\delta_n:C_c(\G^{(n)},\ZZ)\to C_c(\G^{(n-1)},\ZZ)$ given by
  $$\delta_1=s_*-r_*\qquad\text{and}\qquad\delta_n=\sum_{i=0}^n(-1)^nd_{i*}\,.$$
  Then we define the homology $H_*(\G)$ as the homology groups of the chain complex $C_\bullet(\G,\ZZ)$ by
  $$0\longleftarrow C_c(\G^{(0)},\ZZ)\longleftarrow^{\delta_1} C_c(\G^{(1)},\ZZ)\longleftarrow^{\delta_2} C_c(\G^{(2)},\ZZ)\longleftarrow\cdots$$

The following conjecture, posted in \cite{Mat4}, states that the homology of the groupoid refines the $K$-theory of the reduced groupoid $C^*$-algebra.
\newline

\textbf{(HK) conjecture:} Let $\G$ be a minimal, effective, étale groupoid with $\G^{(0)}$ homeomorphic to the Cantor space. Then 
$$K_i(C^*_r(\G))\cong\bigoplus_{n=0}^\infty H_{2n+i}(\G)\,,\qquad\text{for }i=0,1\,.$$

The conjecture was confirmed for the $AF$-groupoids, transformation groupoids of Cantor minimal systems, groupoids of shifts of finite type and products of groupoids of shifts of finite type (see  \cite{Mat2,Mat4}).

Now we are going to collect some results from \cite{Mat2} that will allow us to compute the homology of a groupoid.

Let $\Gamma$ be a countable discrete group and $\G$ an étale groupoid. When $\rho:\G\to\Gamma$ is a groupoid homomorphism, the \emph{skew product} $\G\times_\rho\Gamma$ is $\G\times\Gamma$ with the following groupoid structure: $(g,\gamma)$ and $(g',\gamma')$ are composable if and only if $g$ and $g'$ are composable and $\gamma\rho(g)=\gamma'$, and 
$$(g,\gamma)\cdot(g',\gamma\rho(g))=(gg',\gamma)\qquad\text{and}\qquad(g,\gamma)^{-1}=(g^{-1},\gamma\rho(g))\,.$$
Given $n\in\NN$ we can define the action $\hat{\rho}:\Gamma\curvearrowright(\G\times_\rho \Gamma)^{(n)}$ by 
$$\hat{\rho}^\gamma((g'_1,\gamma'_1),\ldots,(g'_n,\gamma'_n))=((g'_1,\gamma\gamma'_1),\ldots,(g'_n,\gamma\gamma'_n))\,.$$

Two étale groupoids $\G$ and $\HG$ with totally disconnected unit spaces are called \emph{Kakutani equivalent} if there exist full clopen subsets $A$ and $B$ of $\G^{(0)}$ and $\HG^{(0)}$ respectively, such that $\G_A\cong \HG_B$ (\cite[Definition 4.1]{Mat2}). It was proved in \cite[Theorem 3.6(2)]{Mat2} that two Kakutani equivalent groupoids have isomorphic homology groups.

In order to compute the homology of groupoids $\G$ that have a groupoid homomorphism $\rho:\G\to\ZZ$, Matui uses the spectral sequence  
$$E^2_{p,q}=H_p(\ZZ,H_q(\G\times_\rho\ZZ))\Rightarrow H_{p+q}(\G)\,.$$
However, we are going to use a long exact sequence that relates the homology groups of $\G$ and $\G\times_\rho\ZZ$. This sequence might be known to experts but since I could not find any reference we state it for completeness.  I would like to thank Jamie Gabe for suggesting me the use of this long exact sequence.

 \begin{lemma}\label{exact_seq}
 Let $\G$ be an étale groupoid with $\G^{(0)}$ a locally compact, Hausdorff  and totally disconnected space, and let $\rho:\G\to\ZZ$ be a group homomorphism. Then there exists a long exact sequence
 $$\xymatrix {  0  &  H_0(\G)\ar@{>}[l] & H_0(\G\times_\rho\ZZ) \ar@{>}[l] & H_0(\G\times_\rho\ZZ)\ar@{>}[l]_{\Ide-\hat{\rho}^1} & H_1(\G)\ar@{>}[l] & \cdots\ar@{>}[l] \\ \cdots  &  H_n(\G)\ar@{>}[l] & H_n(\G\times_\rho\ZZ) \ar@{>}[l] & H_n(\G\times_\rho\ZZ)\ar@{>}[l]_{\Ide-\hat{\rho}^1} & H_{n+1}(\G)\ar@{>}[l] & \cdots\ar@{>}[l] }\,,$$
 where $\hat{\rho}$ is the induced map by the action $\hat{\rho}:\ZZ\curvearrowright\G\times_\rho \ZZ$.
  \end{lemma} 
  \begin{proof}
  Let $\hat{\rho}^1:\G\times_\rho\ZZ\to \G\times_\rho\ZZ$ given by $g\times\{i\}\mapsto g\times\{i+1\}$ for $g\in \G$ and $i\in\ZZ$, then we define the  short exact sequence
\begin{equation}\label{SES}
\xymatrix {  0 \ar@{>}[r]  &  C_\bullet(\G\times_\rho\ZZ,\ZZ)\ar@{>}[r]^{\Ide-\hat{\rho}^1} & C_\bullet(\G\times_\rho\ZZ,\ZZ) \ar@{>}[r]^{\hat{\pi}} & C_\bullet(\G,\ZZ)\ar@{>}[r] & 0 }\,,
\end{equation}
where $\hat{\pi}$ is induced by the map $\pi:(\G\times_\rho\ZZ)^{(n)} \to  \G^{(n)}$ given by  $$((g_1\times\{i_1\}),\ldots, (g_n\times\{i_n\}))\mapsto (g_1,\ldots,g_n)\,.$$ It is clear that $\Ide-\hat{\rho}^1$ is an injective map, $\hat{\pi}$ is a surjective map and that $\text{im }(\Ide-\hat{\rho}^1)\subseteq \ker \hat{\pi}$. So it is enough to check that $\text{im }\Ide-\hat{\rho}^1\supseteq \ker \hat{\pi}$. Let $f\in C_c((\G\times_\rho \ZZ)^{(n)},\ZZ)$ such that $\hat{\pi}(f)=0$. Observe that we can write $f=\sum_{i=-m}^m f_i$ where $f_i\in C_c(\G^{(n)}\times\{i\},\ZZ)$ such that $\sum_{i=-m}^m\hat{\pi}(f_i)=0$. Let $A$ be the compact support of $f$, and let $B:=\pi(A)$ a compact subset of $\G^{(n)}$. Let $B_1,\ldots,B_k$ be a clopen partition of $B$ such that $(f_i)_{|B_j\times\{i\}}$ is constant for every $-m\leq i\leq m$, so let $\lambda_{i,j}$ be the integer number such that   $(f_i)_{|B_j\times\{i\}}=\lambda_{i,j}$. But then for every $1\leq j\leq k$ we have that $\sum_{i=-m}^m\lambda_{i,j}=0$. We can write $f=\sum_{j=1}^k\sum_{i=-m}^m \lambda_{i,j}\Ind_{B_j\times\{i\}}$, where $\sum_{i=-m}^m\lambda_{i,j}=0$. For any clopen $C$ of $\G^{(n)}$ and $i< j$ we define the function $\g_{i,j,C}:=\sum_{k=i}^{j-1}\Ind_{C\times\{k\}}$. The function $$h:=\sum_{j=1}^k\left( \sum_{i=-m}^{-1} \lambda_{i,j}\g_{i,0,B_j}-\sum_{i=0}^{m-1} \lambda_{i,j}\g_{0,i+1,B_j}\right)\in C_c((\G\times_\rho \ZZ)^{(n)},\ZZ)\,,$$
is such that 
\begin{align*}
(\Ide-\hat{\rho}^1)(h) & =\sum_{j=1}^k\left( \sum_{i=-m}^{-1} \lambda_{i,j}\left( \Ind_{B_j\times\{i\}}-\Ind_{B_j\times\{0\}}\right) -\sum_{i=0}^{m-1} \lambda_{i,j}\left( \Ind_{B_j\times\{0\}}-\Ind_{B_j\times\{i\}}\right)\right) \\
 &  =\sum_{j=1}^k\left( \left( \sum_{i=-m}^{-1} \lambda_{i,j}\Ind_{B_j\times\{i\}}+\sum_{i=1}^m \lambda_{i,j}\Ind_{B_j\times\{i\}}\right)  -\left( \sum_{i=-m}^{-1} \lambda_{i,j}\Ind_{B_j\times\{0\}}+\sum_{i=1}^m \lambda_{i,j}\Ind_{B_j\times\{0\}}\right) \right)  \\
 &  =\sum_{j=1}^k\left( \left( \sum_{i=-m}^{-1} \lambda_{i,j}\Ind_{B_j\times\{i\}}+\sum_{i=1}^m \lambda_{i,j}\Ind_{B_j\times\{i\}}\right)  +\lambda_{0,j}\Ind_{B_j\times\{0\}}  \right)   \\
 &  =\sum_{j=1}^k \sum_{i=-m}^{m} \lambda_{i,j}\Ind_{B_j\times\{i\}}=f\,.  
\end{align*}
Therefore,  $f$ belongs to $\text{im }(\Ide-\hat{\rho}^1)$, as desired.

Then, the long exact sequence of homology of the exact sequence (\ref{SES}), give us the desired sequence  
$$\xymatrix {  0  &  H_0(\G)\ar@{>}[l] & H_0(\G\times_\rho\ZZ) \ar@{>}[l] & H_0(\G\times_\rho\ZZ)\ar@{>}[l]_{\Ide-\hat{\rho}^1} & H_1(\G)\ar@{>}[l] & \cdots\ar@{>}[l] \\ \cdots  &  H_n(\G)\ar@{>}[l] & H_n(\G\times_\rho\ZZ) \ar@{>}[l] & H_n(\G\times_\rho\ZZ)\ar@{>}[l]_{\Ide-\hat{\rho}^1} & H_{n+1}(\G)\ar@{>}[l] & \cdots\ar@{>}[l] }\,.$$
  \end{proof}

The following Lemma is straightforward to prove (see for example \cite[Proposition 4.7]{FKPS}).  
\begin{lemma}\label{inductive}
Let $\G$ be a locally compact, étale groupoid with $\G^{(0)}$ a totally disconnected Hausdorff space. Let us suppose that there exists a sequence $\G_1,\G_2,\ldots$ of open subgroupoids of $\G$, such that $\G_i\subseteq \G_{i+1}$ with $\bigcup_{i=1}^\infty \G_i=\G$. Then $H_*(\G)\cong \varinjlim   H_*(\G_i)$ where the maps $H_*(\G_i)\to H_*(\G_{i+1})$ are induced by the natural inclusions $\G_i\hookrightarrow \G_{i+1}$.
\end{lemma}

\section{The Katsura groupoid.}

Let $N\in \NN\cup\{\infty\}$, and let $\Asf$ and $\Bsf$ be two $N\times N$ row-finite matrices with integer entries, and such that $\Asf_{i,j}\geq 0$ for all $i$ and $j$. We define 
$$\Omega_A:=\{(i,j)\in \{1,\ldots,N\}^2:\Asf_{i,j}\neq 0\}\,.$$
We say that $\Asf$ is \emph{row-finite} if  $\sum_{j=1}^N|\Asf_{i,j}|<\infty$ for every $1\leq i\leq N$.

Throughout the paper we will assume that $\Asf$ and $\Bsf$ are row finite matrices with no identically zero rows. Let $\Egr$ be the graph with $\Egr^0=\{1,\ldots,N\}$, and such that the set of edges from vertex $i$ to vertex $j$ is a set of $\Asf_{i,j}$ elements, say 
$$\Egr^1:=\{e_{i,j,n}:0\leq n<\Asf_{i,j}\}\,,$$
with source map given by $s(e_{i,j,n})=i$ and range map by $r(e_{i,j,n})=j$. A path of length $n$, is a concatenation of edges $\alpha_1\cdots\alpha_n$ with $r(\alpha_i)=s(\alpha_{i+1})$, and we denote by $\Egr^n$ the set of all paths of length $n$. Given a path $\alpha$ we denote by $|\alpha|$ its length.
Since by assumption $\Asf$ has no identically zero rows, $\Egr$ has no sinks.  We denote by $\Egr^*$ the set of all finite length paths (including the length zero paths, which are the vertices), and let $\Egr^\infty=\{(\alpha_i)\in \prod_{i=1}^\infty \Egr^1: r(\alpha_i)=s(\alpha_{i+1}) \}$ be the infinity path space with the product and subspace topology. So given a finite path $\alpha\in \Egr^*$ we define the compact and open set $Z(\alpha):=\{\alpha x: x\in \Egr^\infty \text{ with }s(x)=r(\alpha)\}$.  This family of sets is a basis of clopen and compact subsets, and hence $\Egr^\infty$ is a totally disconnected and locally compact space. Observe that $\Egr^\infty$ is compact if and only if $N\in\NN$.

Given $\alpha\in \Egr^1$ we define $\Asf_{\alpha}:=\Asf_{s(\alpha),r(\alpha)}$ and $\Bsf_{\alpha}:=\Bsf_{s(\alpha),r(\alpha)}$. Then given $y=y_1 y_2\cdots \in \Egr^\infty$ with $y_i\in \Egr^1$  and  $n\in\NN$ we define $y_{|n}=y_1\cdots y_n\in \Egr^*$, 
 $\Asf_{y_{|n}}:=\Asf_{y_1}\cdots \Asf_{y_n}$ and $\Bsf_{y_{|n}}:=\Bsf_{y_1}\cdots \Bsf_{y_n}$.

We define an action $\kappa$ of $\ZZ$ on $\Egr$ which is trivial in $\Egr^0$, and which acts on edges as follows: given $m\in\ZZ$, and $e_{i,j,n}\in \Egr^1$, let $(k,l)$ be the unique pair of integers such that 
$$m\Bsf_{i,j}+n=k\Asf_{i,j}+l\qquad\text{and}\qquad 0\leq l<\Asf_{i,j}\,.$$
We then define 
$$\kappa_m(e_{i,j,n})=e_{i,j,l}\,.$$

Moreover we define the map $\varphi:\ZZ\times \Egr^1\to \ZZ$ as $\varphi(m,e_{i,j,n})=k$, that satisfies
$$\varphi(m_1+m_2,e)=\varphi(m_1,\kappa_{m_2}(e))+\varphi(m_2,e)\,,$$
for every $e\in\Egr^1$ and $m_1,m_2\in\ZZ$. A map satisfying the above equality is called a \emph{one-cocyle for the action $\kappa$}.  

The action $\kappa$ is  \emph{pseudo free} if given $m\in \ZZ$ and $e\in \Egr^1$, whenever $\kappa_m(e)=e$ and $\varphi(m,e)=0$, then $m=0$.
We will say that $\Asf$ and $\Bsf$ is a \emph{pseudo free} pair of matrices if the action $\kappa$ is pseudo free. By \cite[Lemma 18.5]{EP} the pair $\Asf$ and $\Bsf$ is pseudo free whenever 
$$\Bsf_{i,j}=0\qquad \text{if and only if}\qquad \Asf_{i,j}=0\,.$$

Then we can extend by induction  the action  $\kappa$ and the one-cocycle $\varphi$ to  paths of arbitrary length \cite[Proposition 2.4]{EP}: Assume that $n\geq 1$ and the action $\kappa$ on $E^n$ and the one-cocycle $\varphi:\ZZ\times \Egr^n\to \ZZ$ for $\kappa$ are defined. Given $\alpha'\in \Egr^1$, $\alpha''\in \Egr^n$ with $r(\alpha')=s(\alpha'')$ and $m\in \ZZ$, we define
$$\kappa_m(\alpha'\alpha''):=\kappa_m(\alpha')\kappa_{\varphi(m,\alpha')}(\alpha'') \qquad \text{and}\qquad \varphi(m,\alpha'\alpha'')=\varphi(\varphi(m,\alpha'),\alpha'')\,.$$
In particular $\kappa$ can be extended to an action of $\Egr^\infty$.

Now we denote by $\SAB$ the set of triples $(\alpha,m,\beta)$ where $\beta,\alpha\in \Egr^*$ with $r(\alpha)=r(\beta)$ and $m\in\ZZ$. In \cite{EP} the set $\SAB$ was given the structure of inverse $*$-semigroup, and the groupoid of a certain partial action of $\SAB$ on $\Egr^\infty$ was constructed. Here we will avoid to explain all the construction and defined only the resulting groupoid.

We define the equivalence relation on the set of quadruples of the form $(\alpha,m,\beta;x)$ where $(\alpha,m,\beta)\in \SAB$ and $x\in Z(\beta)$ generated by the relation: 
$$(\alpha,m,\beta;x)\sim (\alpha\kappa_m(\gamma),\varphi(m,\gamma),\beta\gamma; x)\,,$$
where $x=\beta \gamma y$ for $\gamma\in \Egr^1$ with $s(\gamma)=r(\beta)$ and $y\in \Egr^\infty$ with $s(y)=r(\gamma)$. We denote by $[\alpha,n,\beta; x]$ the equivalence class under the above equivalent relation.

Then given a pseudo free pair $\Asf$ and $\Bsf$ of $N\times N$ matrices, we define the \emph{Katsura-Exel-Pardo groupoid} 
$$\G_{\Asf,\Bsf}:=\{[\alpha,m,\beta; x]: (\alpha,m,\beta)\in\SAB\text{ and }x\in Z(\beta) \}\,,$$
with  product defined by $$[\alpha,m',\beta;\beta\kappa_{m}(x)]\cdot [\beta,m,\gamma;\gamma x]=[\alpha,m'+m,\gamma;x]\,,$$
and inverse
$$[\alpha,m,\beta; x]^{-1}=[\beta,-m,\alpha;\alpha\kappa_m(y)]\,,\qquad\text{if }x=\beta y\,.$$
Therefore if we identify $\G_{\Asf,\Bsf}^{(0)}$ with $\Egr^\infty$ via the map $[v,0,v;x]\mapsto x$ for $v\in \Egr^0$ and $x\in Z(v)$, then the range and the source map can be defined as
$$s([\alpha,m,\beta; x])=x\qquad \text{and}\qquad r([\alpha,m,\beta; x])=\alpha\kappa_m(y)\,,\qquad \text{if }x=\beta y\,.$$ 
With the topology given by the set of open and compact subsets 
$$Z(\alpha,m,\beta;U):=\{[\alpha,m,\beta; x]:  x\in U\}\,,$$
where $U$ is an open and compact subset of $Z(\beta)$, the groupoid $\G_{\Asf,\Bsf}$ is étale  with unit space $\G_{\Asf,\Bsf}^{(0)}$ a locally compact, totally disconnected space. The sets of the form $Z(\alpha,m,\beta;Z(\beta))$ forms a basis for the topology. Indeed,  any open subset $U\subseteq Z(\beta)$ can be written as a disjoint union $\bigsqcup_i Z(\beta\gamma_i)$ where $s(\gamma_i)=r(\beta)$ for every $i$. Whence,
$$Z(\alpha,m,\beta;U)=\bigsqcup_i Z(\alpha,m,\beta;Z(\beta\gamma_i))=\bigsqcup_i Z(\alpha\kappa_m(\gamma_i),\varphi(m,\gamma_i),\beta\gamma_i;Z(\beta\gamma_i))\,.$$

In \cite{EP} and later in \cite{EPS} it was shown that $C^*(\G_{\Asf,\Bsf})$ is isomorphic to the $C^*$-algebra $\mathcal{O}_{\Asf,\Bsf}$ constructed in \cite{Kat}. 

Now we summarize the properties of the groupoid $\G_{\Asf,\Bsf}$  (see \cite[Section 18]{EP}):
\begin{enumerate}
	\item $\G_{\Asf,\Bsf}$ is an étale, locally compact, amenable groupoid,  
    \item $\G_{\Asf,\Bsf}^{(0)}$ is a locally compact, totally disconnected Hausdorff space, and it is compact if and only if $N<\infty$, 
	\item  $\G_{\Asf,\Bsf}$ is effective  if 
	\begin{enumerate}
	\item every circuit in $\Egr$ has an exit,
	\item for every $1\leq i\leq N$ there exists $x\in Z(i)$ such that  $\lim\limits_{n\to\infty}\frac{\Bsf_{x_{|n}}}{\Asf_{x_{|n}}}=0$,
	\end{enumerate}
	\item if the matrix $\Asf$ is irreducible and it is not a permutation matrix, then $\G_{\Asf,\Bsf}$ is minimal and purely infinite \cite[Definition 4.9]{Mat3}.

\end{enumerate}
In \cite[Theorem 18.6]{EP} there were given additional conditions for $\G_{\Asf,\Bsf}$ being a Hausdorff groupoid. In particular, when $\Asf$ and $\Bsf$ are pseudo free the groupoid $\G_{\Asf,\Bsf}$ is Hausdorff. Katsura showed that 
$$K_0(C^*(\G_{\Asf,\Bsf}))\cong \coker (\Ide-\Asf)\oplus \ker (\Ide-\Bsf)\qquad\text{and}\qquad K_1(C^*(\G_{\Asf,\Bsf}))\cong \coker (\Ide-\Bsf)\oplus \ker (\Ide-\Asf)\,,$$ and that given two countably generated abelian groups $G_0$ and $\G_1$  there exists an irreducible matrix $\Asf$ and a matrix $\Bsf$ satisfying condition 
\begin{equation*}
(O)\qquad \Asf_{i,i}\geq 2\qquad \text{and}\qquad \Asf_{i,i}> |\Bsf_{i,i}| \text{ for every }i\,,
\end{equation*} 
such that $G_0\cong \coker (\Ide-\Asf)\oplus \ker (\Ide-\Bsf)$ and $G_1\cong \coker (\Ide-\Bsf)\oplus \ker (\Ide-\Asf)$ \cite[Proposition 4.5]{Kat2}, and hence $\G_{\Asf,\Bsf}$  is an effective, minimal and purely infinite groupoid.

We define the  homomorphism $\rho:\G_{\Asf,\Bsf}\to \ZZ$ given by $[\alpha,n,\beta;x]\mapsto |\alpha|-|\beta|$, and we define the subgropoid $\HG_{\Asf,\Bsf}:=\ker \rho$. 

\begin{lemma}[{cf. \cite[Lemma 4.13]{Mat2}}]\label{Morita}
	Let $N\in \NN\cup\{\infty\}$, and let $\Asf$ and $\Bsf$ be a pair of pseudo free $N\times N$  row-finite matrices with integer entries, and such that $\Asf_{i,j}\geq 0$ for all $i$ and $j$. Let $\rho:\G_{\Asf,\Bsf}\to \ZZ$ be the above defined homomorphism, and let $Y=\G_{\Asf,\Bsf}^{(0)}\times \{0\}$. Then $Y$ is a $(\G_{\Asf,\Bsf}\times_\rho\ZZ)$-full open subspace of $(\G_{\Asf,\Bsf}\times_\rho\ZZ)^{(0)}$ and $(\G_{\Asf,\Bsf}\times_\rho\ZZ)_Y\cong \ker \rho$. In particular $\G_{\Asf,\Bsf}\times_\rho\ZZ$ is Kakutani equivalent to $\HG_{\Asf,\Bsf}$.
\end{lemma}

Now given $n\in\NN$ we define the open subgroupoid 
$$\HG_{\Asf,\Bsf,n}:=\{[\alpha,m,\beta;x]\in \HG_{\Asf,\Bsf}: |\alpha|=|\beta|=n\}\,,$$
and then the map $\eta_n:\HG_{\Asf,\Bsf,n}\to\ZZ$ given by $[\alpha,m,\beta;x]\mapsto m$ is a well-defined groupoid homomorphism. Since $\Asf$ has no zero rows, given $\beta,\alpha\in \Egr^*$ with $r(\alpha)=r(\beta)$,  $m\in\ZZ$ we have that
$$[\alpha,m,\beta;x]=[\alpha\kappa_m(\gamma),\varphi(m,\gamma),\beta\gamma;x]\,,$$
if $x=\beta\gamma y$ for some $y\in \Egr^\infty$ and $\gamma\in \Egr^1$ with $s(\gamma)=r(\beta)$ and $s(y)=r(\gamma)$, then it follows that $\HG_{\Asf,\Bsf,n} \subseteq \HG_{\Asf,\Bsf,n+1}$ for every $n\in\NN$, moreover $\HG_{\Asf,\Bsf}=\bigcup_{n=0}^\infty\HG_{\Asf,\Bsf,n}$.

We define the groupoid $\RG_{\Asf,\Bsf,n} :=\ker \eta_{n}$, which is Kakutani equivalent to  $\HG_{\Asf,\Bsf,n}\times_{\eta_n} \ZZ$ (Lemma \ref{Morita}).

\begin{lemma}\label{lemma_HG}
Let $N\in \NN\cup\{\infty\}$, and let $\Asf$ and $\Bsf$ be a pair of pseudo free  $N\times N$ row finite matrices with integer entries, and such that $\Asf_{i,j}\geq 0$ for all $i$ and $j$. For $\HG_{\Asf,\Bsf}=\ker \rho$, we have that  $H_i(\HG_{\Asf,\Bsf})=0$ for $i\geq 2$.
\end{lemma}
\begin{proof}
  Given $n\in\NN$, we claim $\RG_{\Asf,\Bsf,n}$ is an $AF$ groupoid. Indeed, $\G_{\Asf,0}$ is the graph groupoid of $\Egr$,  and it is well-known that $\HG_{\Asf,0}$ is an $AF$-groupoid (see for example \cite[Proposition 6.1]{FKPS}). But $\RG_{\Asf,\Bsf,n}\subseteq \HG_{\Asf,0}$ is an open subgroupoid, so $\RG_{\Asf,\Bsf,n}$ is an $AF$-groupoid as well.  
  
 Since $\RG_{\Asf,\Bsf,n}$ is an $AF$ groupoid we have that  $H_i(\HG_{\Asf,\Bsf,n}\times_{\eta_n}\ZZ)\cong H_i(\RG_{\Asf,\Bsf,n})=0$ for $i\geq 1$ (\cite[Theorem 4.11]{Mat2}), so  using Lemma \ref{exact_seq} we have the exact sequences 
$$0\longrightarrow H_1(\HG_{\Asf,\Bsf,n})\longrightarrow H_0(\HG_{\Asf,\Bsf,n}\times_{\eta_n}\ZZ)\longrightarrow H_0(\HG_{\Asf,\Bsf,n}\times_{\eta_n}\ZZ)\longrightarrow H_0(\HG_{\Asf,\Bsf,n})\longrightarrow0\,, $$
and 
$$0\longrightarrow H_i(\HG_{\Asf,\Bsf,n})\longrightarrow 0\,,$$ for $i\geq 2$. Therefore by Lemma \ref{inductive} it follows that $H_i(\HG_{\Asf,\Bsf})=\varinjlim H_i(\HG_{\Asf,\Bsf,n})=0$ for $i\geq 2$.
\end{proof}

Now we are going to give an explicit computation of the lower degree homology groups of $\HG_{\Asf,\Bsf}$. 

Given groups  $G_1,G_2,G_3,\ldots$ and maps $\varphi_{i,i+1}:G_i\to G_{i+1}$, we denote by $\varinjlim (G_i,\varphi_{i,i+1})$ its inductive limit, and the maps $\varphi_{i,\infty}:G_i\to \varinjlim (G_i,\varphi_{i,i+1})$ the canonical ones.


We write by $\ZZ_\Asf$ the abelian group given by the inductive limit $\varinjlim (\ZZ^{\Egr^0},\varphi^\Asf_{i,i+1})$, where the maps $\varphi_{i,i+1}^\Asf:\ZZ^{\Egr^0}\to \ZZ^{\Egr^0}$ are given by $\Ind_v\mapsto \sum_{w\in \Egr^0} |v\Egr^1w| \Ind_w=\sum_{w\in \Egr^0} \Asf_{v,w} \Ind_w$.

\begin{proposition}\label{Hom1} Let $N\in \NN\cup\{\infty\}$, and let $\Asf$ and $\Bsf$ be a pair of pseudo free $N\times N$  row-finite matrices with integer entries, and such that $\Asf_{i,j}\geq 0$ for all $i$ and $j$. Then  there exists a group isomorphism $\Phi_\Asf:H_0(\HG_{\Asf,\Bsf})\to \ZZ_\Asf$ given by the map $$[\Ind_{Z(\alpha,0,\alpha;Z(\alpha))}]\mapsto\varphi^\Asf_{n,\infty}(\Ind_v)\,,$$ where  $\alpha\in \Egr^n$ and $r(\alpha)=v$.
\end{proposition}
\begin{proof}
	Let $\Phi_\Asf:H_0(\HG_{\Asf,\Bsf})\to \ZZ_\Asf$ be the above defined map. 
	First recall that the boundary map $\delta_1:C_c(\HG_{\Asf,\Bsf},\ZZ)\to C_c(\HG^{(0)}_{\Asf,\Bsf},\ZZ)$ sends $\Ind_{Z(\alpha,m,\beta;Z(\alpha))}\mapsto \Ind_{Z(\beta,0,\beta;Z(\beta))}-\Ind_{Z(\alpha,0,\alpha;Z(\alpha))}$ for $\alpha,\beta\in \Egr^*$ with $|\alpha|=|\beta|$ and $r(\alpha)=r(\beta)$, and $m\in\ZZ$. Therefore $[\Ind_{Z(\alpha,0,\alpha;Z(\alpha))}]=[\Ind_{Z(\beta,0,\beta;Z(\beta))}]\in H_0(\HG_{\Asf,\Bsf})$ whenever $|\alpha|=|\beta|$ and $r(\alpha)=r(\beta)$. But now given $\alpha\in \Egr^n$ with $r(\alpha)=v$ we have that  
	\begin{align*} \varphi^\Asf_{n,\infty}(\Ind_v) & =\Phi_\Asf([\Ind_{Z(\alpha,0,\alpha;Z(\alpha))}])=\Phi_\Asf\left( \sum_{\beta\in v\Egr^1}[\Ind_{Z(\alpha\beta,0,\alpha\beta;Z(\alpha\beta))}]\right) \\
	 & =  \sum_{\beta\in v\Egr^1}\Phi_\Asf\left([\Ind_{Z(\alpha\beta,0,\alpha\beta;Z(\alpha\beta))}]\right) = \sum_{\beta\in v\Egr^1} \varphi^\Asf_{n+1,\infty}(\Ind_{r(\beta)}) \\ &= \varphi^\Asf_{n+1,\infty}\left( \sum_{w\in \Egr^0}|v\Egr^1w|\Ind_w\right) =\varphi^\Asf_{n,\infty}(\Ind_v)\,.
	\end{align*}
	
	This shows that $\Phi_\Asf$ is well-defined but also its inverse map $\Phi_\Asf^{-1}$, so $\Phi_\Asf$ is an isomorphism.
\end{proof}

We write by $\ZZ_\Bsf$ the abelian group given by the inductive limit $\varinjlim (\ZZ^{\Egr^0},\varphi^\Bsf_{i,i+1})$, where the maps $\varphi_{i,i+1}^\Bsf:\ZZ^{\Egr^0}\to \ZZ^{\Egr^0}$ are given by $\Ind_v\mapsto \sum_{w\in \Egr^0} \Bsf_{v,w} \Ind_w$.

\begin{proposition}\label{hom2} Let $N\in \NN\cup\{\infty\}$, and let $\Asf$ and $\Bsf$ be a pair of pseudo free $N\times N$  row-finite matrices with integer entries, and such that $\Asf_{i,j}\geq 0$ for all $i$ and $j$.  Then there exists a group isomorphism $\Phi_\Bsf:H_1(\HG_{\Asf,\Bsf})\to \ZZ_\Bsf$ given by the map $$[\Ind_{Z(\alpha,1,\alpha;Z(\alpha))}]\mapsto\varphi^\Bsf_{n,\infty}(\Ind_v)$$ where  $\alpha\in \Egr^n$ and $r(\alpha)=v$.
\end{proposition}
\begin{proof}
	Let $\Phi_\Bsf:H_1(\HG_{\Asf,\Bsf})\to \ZZ_\Bsf$ be the above defined map. First recall that the boundary map $\delta_2:C_c(\HG^{(2)}_{\Asf,\Bsf},\ZZ)\to C_c(\HG_{\Asf,\Bsf},\ZZ)$ sends 
	$$\Ind_{Z(\alpha,m,\beta;Z(\beta))\times Z(\beta,n,\gamma;Z(\gamma))\cap \HG_{\Asf,\Bsf}^{(2)}}\mapsto \Ind_{Z(\beta,n,\gamma;Z(\gamma))}-\Ind_{Z(\alpha,n+m,\gamma;Z(\gamma))}+\Ind_{Z(\alpha,m,\beta;Z(\beta))}$$ for $\alpha,\beta,\gamma\in \Egr^*$ with $|\alpha|=|\beta|=|\gamma|$ and $r(\alpha)=r(\beta)=r(\gamma)$, and $m,n\in\ZZ$.
	In particular we have that 
	\begin{align*} 
	[\Ind_{Z(\alpha,m,\beta;Z(\beta))}]  & =[\Ind_{Z(\alpha,m,\alpha;Z(\alpha))}]+[\Ind_{Z(\alpha,0,\beta;Z(\beta))}]\,,\\
	[\Ind_{Z(\alpha,m,\beta;Z(\beta))}]  & =[\Ind_{Z(\beta,m,\beta;Z(\alpha))}]-[\Ind_{Z(\beta,0,\alpha;Z(\alpha))}]\,,\\
	[\Ind_{Z(\alpha,0,\beta;Z(\beta))}] & = [\Ind_{Z(\alpha,0,\gamma;Z(\gamma))}] + [\Ind_{Z(\gamma,0,\beta;Z(\beta))}]\,,
	\end{align*}
	in $C_c(\HG_{\Asf,\Bsf},\ZZ)/\text{im}(\delta_2)$, from where we can deduce that
	\begin{align*}
	[\Ind_{Z(\alpha,0,\alpha;Z(\alpha))}] & =0 \\
	[\Ind_{Z(\alpha,0,\beta;Z(\beta))}] & = -[\Ind_{Z(\beta,0,\alpha;Z(\alpha))}]\,, \\
	[\Ind_{Z(\alpha,m,\alpha;Z(\alpha))}] & =[\Ind_{Z(\beta,m,\beta;Z(\beta))}]\,, \\
		[\Ind_{Z(\alpha,m,\alpha;Z(\alpha))}] & =m\cdot[ \Ind_{Z(\alpha,1,\alpha;Z(\alpha))}] \,.	
	\end{align*}
		in $C_c(\HG_{\Asf,\Bsf},\ZZ)/\text{im}(\delta_2)$. Let $\alpha,\beta\in \Egr^n$ with $r(\alpha)=r(\beta)$, then we have that
		$$Z(\alpha,1,\beta;Z(\beta))=\bigsqcup_{e\in r(\alpha)\Egr^1}Z(\alpha\kappa_1(e),\varphi(1,e),\beta e;Z(\beta e))\,.$$

		Now let $f\in C_c(\HG_{\Asf,\Bsf},\ZZ)$ with $\delta_1(f)=0$, then by the above we can assume that 
		$$f=\sum_{i=1}^k\lambda_i\cdot\Ind_{Z(\alpha_i,m_i,\beta_i;Z(\beta_i))}\,,$$
		with $\alpha_i,\beta_i\in \Egr^n$ for some $n\in\NN$, and $m_i\in\ZZ$. For every $v\in \Egr^0$ we choose a $\alpha_v\in \Egr^n$ with $r(\alpha_v)=v$, then by the above relations we can assume that 
		$$f=\sum_{v\in \Egr^0}\lambda_v \Ind_{Z(\alpha_v,1,\alpha_v;Z(\alpha_v))}+\sum_{v\in \Egr^0}\sum_{\gamma \in  \Egr^nv\setminus \{\alpha_v\}}\xi_\gamma \Ind_{Z(\gamma,0,\alpha_v;Z(\alpha_v))}\,.$$
		But then 
		\begin{align*}
		\delta_2(f) & =\delta_2\left( \sum_{v\in \Egr^0}\sum_{\gamma \in  \Egr^nv\setminus \{\alpha_v\}}\xi_\gamma \Ind_{Z(\gamma,0,\alpha_v;Z(\alpha_v))}\right) \\
		& =\sum_{v\in \Egr^0}\sum_{\gamma \in  \Egr^nv\setminus \{\alpha_v\}}\xi_\gamma \left(  \Ind_{Z(\alpha_v,0,\alpha_v;Z(\alpha_v))}-\Ind_{Z(\gamma,0,\gamma;Z(\gamma))}\right) =0\,,
		\end{align*}
		but this implies that $\xi_\gamma=0$ for every $\gamma\in \Egr^n$. Thus we can assume that 
		$$f=\sum_{v\in \Egr^0}\lambda_v \Ind_{Z(\alpha_v,1,\alpha_v;Z(\alpha_v))}\,.$$
	
		Then if for every $w\in \Egr^0$ we choose $\beta_w\in \Egr^{n+1}$ with $r(\beta_w)=w$, and because of $\sum_{e\in v\Egr^1w}\varphi(1,e)=\Bsf_{v,w}$, we have that 
		\begin{align*}
        \varphi^\Bsf_{n,\infty}\left( \sum_{v\in\Egr^0} \lambda_v \Ind_v\right) &= \Phi_{\Bsf}([f])=\Phi_{\Bsf}\left(\left[  \sum_{v\in \Egr^0}\lambda_v \Ind_{Z(\alpha_v,1,\alpha_v;Z(\alpha_v))}\right] \right) 
        \\ & =\Phi_\Bsf\left(\left[  \sum_{v\in \Egr^0}\lambda_v \sum_{w\in \Egr^0}\sum_{e\in v\Egr^1w} \Ind_{Z(\alpha_v\kappa_1(e),\varphi(1,e),\alpha_ve;Z(\alpha_ve))}\right]\right) \\
           & =\Phi_\Bsf\left(\sum_{v\in \Egr^0}\lambda_v \sum_{w\in \Egr^0}\sum_{e\in v\Egr^1w}\left[ \Ind_{Z(\beta_w,\varphi(1,e),\beta_w;Z(\beta_w))}\right]\right) \\
                      & =\Phi_\Bsf\left(\sum_{v\in \Egr^0}\lambda_v \sum_{w\in \Egr^0}\sum_{e\in v\Egr^1w}\varphi(1,e)\left[  \Ind_{Z(\beta_w,1,\beta_w;Z(\beta_w))}\right]\right) \\
                                            & =\Phi_\Bsf\left(\sum_{v\in \Egr^0}\lambda_v \sum_{w\in \Egr^0}\Bsf_{v,w}\left[  \Ind_{Z(\beta_w,1,\beta_w;Z(\beta_w))}\right]\right) \\
                                            &=\varphi^\Bsf_{n+1,\infty}\left(\sum_{v\in \Egr^0}\lambda_v \sum_{w\in \Egr^0}\Bsf_{v,w}  \Ind_w\right) \,.
		\end{align*}

		Now let $\alpha\in\Egr^n$ and let $m\in\ZZ$. Then $Z(\alpha,m,\alpha;Z(\alpha))=Z(\alpha,0,\alpha;Z(\alpha))$ if and only if for every $x\in Z(r(\alpha))$ there exists $k\in\NN$ such that $\Bsf_{x_{|k}}=0$ if and only if $\varphi^\Bsf_{n,\infty}(\Ind_{r(\alpha)})=0$.
		Therefore, $\Phi_\Bsf$ is a well-defined map. Clearly $\Phi_\Bsf$ is an surjective map, and injectivity follows since by the above argument the inverse map $\Phi_\Bsf^{-1}$ is also well-defined. 
\end{proof}

Now using Lemma \ref{exact_seq} we have the following long exact sequence
\begin{equation}\label{SES_Hom}
\xymatrix {  0 &  H_0(\G_{\Asf,\Bsf})\ar@{>}[l] & H_0(\G_{\Asf,\Bsf}\times_\rho\ZZ) \ar@{>}[l] & H_0(\G_{\Asf,\Bsf}\times_\rho\ZZ)\ar@{>}[l]_{\Ide-\hat{\rho}^1} &  \ar@{>}[l]\\ 
	{}  &  H_1(\G_{\Asf,\Bsf})\ar@{>}[l] & H_1(\G_{\Asf,\Bsf}\times_\rho\ZZ) \ar@{>}[l] & H_1(\G_{\Asf,\Bsf}\times_\rho\ZZ)\ar@{>}[l]_{\Ide-\hat{\rho}^1} & H_{2}(\G_{\Asf,\Bsf})\ar@{>}[l] & 0\ar@{>}[l] \,,}
\end{equation}
and $H_i(\G_{\Asf,\Bsf})=0$ for $i\geq 3$.

It is then enough to describe the action $\hat{\rho}:\ZZ \curvearrowright H_i(\G_{\Asf,\Bsf}\times_\rho\ZZ))$ for $i=0,1$. Observe that $Y:=\G_{\Asf,\Bsf}^{(0)}\times\{0\}$ is a full open subset, and  that $(\G_{\Asf,\Bsf}\times_\rho\ZZ)_{Y}\cong \HG_{\Asf,\Bsf}$. Then for every $x\in H_i(\G_{\Asf,\Bsf}\times_\rho\ZZ)$ there exists $f\in C_c(\HG^{(i)}_{\Asf,\Bsf},\ZZ)$ such that $[f]=x$ in  $H_i(\G_{\Asf,\Bsf}\times_\rho\ZZ)$, so the assignment $x\to [f]$ gives the group isomorphism $\Psi:H_i(\G_{\Asf}\times_\rho\ZZ)\to H_i(\HG_{\Asf,\Bsf})$. Then the action $\tilde{\rho}:\ZZ\curvearrowright H_i(\HG_{\Asf,\Bsf})$ is defined as the unique action that makes the diagram 
$$\xymatrix {   H_i(\G_{\Asf,\Bsf}\times_\rho\ZZ)\ar@{>}[r]^{ \hat{\rho}^1}\ar@{>}[d]_{\Psi} &  H_i(\G_{\Asf,\Bsf}\times_\rho\ZZ)\ar@{>}[d]_{\Psi}\\
           H_i(\HG_{\Asf,\Bsf})\ar@{>}[r]^{\tilde{\rho}^1 } & H_i(\HG_{\Asf,\Bsf}) }$$
commutative.

\begin{proposition}\label{trans_map}
Let $N\in \NN\cup\{\infty\}$, and let $\Asf$ and $\Bsf$ be a pair of pseudo free $N\times N$  row-finite matrices with integer entries, and such that $\Asf_{i,j}\geq 0$ for all $i$ and $j$.  Then $\Phi_\Asf\circ\tilde{\rho}^1\circ\Phi_\Asf^{-1}:\ZZ_\Asf\to\ZZ_\Asf$ is given by $\varphi^\Asf_{i,\infty}(x))\mapsto\varphi^\Asf_{i+1,\infty}(x)$ for every $x\in\ZZ^N$. Moreover, $\Phi_\Bsf\circ\tilde{\rho}^1\circ\Phi_\Bsf^{-1}:\ZZ_\Bsf\to\ZZ_\Bsf$ is given by $\varphi^\Bsf_{i,\infty}(x)\mapsto\varphi^\Bsf_{i+1,\infty}(x)$ for every $x\in\ZZ^N$. 
\end{proposition}
\begin{proof}
	First recall that the homeomorphism  $\hat{\rho}^1:\G_{\Asf,\Bsf}\times_{\rho}\ZZ\to \G_{\Asf,\Bsf}\times_{\rho}\ZZ$ is given by $g\times\{k\}\to g\times\{k+1\}$ for $g\in \G_{\Asf,\Bsf}$ and $k\in\ZZ$. Now let $Z(\alpha,m,\beta;Z(\beta))\times\{0\}$ be a clopen bisection of $(\G_{\Asf,\Bsf}\times_\rho\ZZ)_{Y}\cong \HG_{\Asf,\Bsf}$, then $\hat{\rho}^1(Z(\alpha,m,\beta;Z(\beta))\times\{0\})=Z(\alpha,m,\beta;Z(\beta))\times\{1\}\subseteq \G_{\Asf,\Bsf}\times_{\rho}\ZZ$, so the induced map 
	$$\hat{\rho}^1:C_c(\HG_{\Asf,\Bsf},\ZZ)\to C_c(\G_{\Asf,\Bsf}\times_{\rho}\ZZ,\ZZ)\qquad\text{is given by}\qquad \Ind_{Z(\alpha,m,\beta;Z(\beta))\times\{0\}}\mapsto \Ind_{Z(\alpha,m,\beta;Z(\beta))\times\{1\}}\,.$$
	Thus, we need to find the equivalent function of $\Ind_{Z(\alpha,m,\beta;Z(\beta))\times\{1\}}$ in $C_c(\HG_{\Asf,\Bsf},\ZZ)$. 
	First observe that 
	$$\hat{\rho}^1:C_c(\HG^{(0)}_{\Asf,\Bsf},\ZZ)\to C_c((\G_{\Asf,\Bsf}\times_{\rho}\ZZ)^{(0)},\ZZ)\qquad\text{is given by}\qquad \Ind_{Z(\alpha,0,\alpha;Z(\alpha))\times\{0\}}\mapsto \Ind_{Z(\alpha,0,\alpha;Z(\alpha))\times\{1\}}\,,$$
	and that given any  $\beta\in \Egr^1$ with $s(\alpha)=r(\beta)$ we have that 
	$$\delta_1(\Ind_{Z(\beta\alpha,0,\alpha;Z(\alpha))\times\{0\}})=\Ind_{Z(\alpha,0,\alpha;Z(\alpha))\times\{1\}}-\Ind_{Z(\beta\alpha,0,\beta\alpha;Z(\alpha))\times\{0\}}\,,$$ 
	so $\left[ \Ind_{Z(\alpha,0,\alpha;Z(\alpha))\times\{1\}}\right] =\left[ \Ind_{Z(\beta\alpha,0,\beta\alpha;Z(\alpha))\times\{0\}}\right]$ in $H_0(\G_{\Asf,\Bsf}\times_{\rho}\ZZ)$. Then,
	
	\begin{align*}
	\tilde{\rho}^1 (\Phi_\Asf^{-1}(\varphi^\Asf_{i,\infty}(\Ind_v)))& = \hat{\rho} \left(\left[  \Ind_{Z(\alpha,0,\alpha;Z(\alpha))\times\{0\}} \right]\right)  = \left[  \Ind_{Z(\alpha,0,\alpha;Z(\alpha))\times\{1\}} \right] \\ & =  \left[ \Ind_{Z(\beta\alpha,0,\beta\alpha;Z(\beta\alpha))\times\{0\}}\right] = \Phi^{-1}_\Asf\left( \varphi^\Asf_{i+1,\infty}(\Ind_v)\right) \,,
	\end{align*} 
	as desired.
	
	Now, on the other hand given any $\alpha\in \Egr^n$ and any  $\beta\in \Egr^1$ with $r(\beta)=s(\alpha)$, we can define the functions in $C_c((\G_{\Asf,\Bsf}\times_\rho\ZZ)^{(2)},\ZZ)$
	$$f_1=\Ind_{(Z(\alpha,1,\alpha;Z(\alpha))\times\{1\})\times (Z(\alpha,0,\beta\alpha;Z(\beta\alpha))\times\{1\})\cap (\G_{\Asf,\Bsf}\times_\rho\ZZ)^{(2)}}\,,$$
	$$f_2=\Ind_{(Z(\beta\alpha,0,\alpha;Z(\alpha))\times\{0\})\times (Z(\alpha,1,\beta\alpha;Z(\beta\alpha))\times\{1\})\cap(\G_{\Asf,\Bsf}\times_\rho\ZZ)^{(2)}}\,,$$
	$$f_3=\Ind_{(Z(\beta\alpha,0,\alpha;Z(\alpha))\times\{0\})\times (Z(\alpha,0,\beta\alpha;Z(\beta\alpha))\times\{1\})\cap(\G_{\Asf,\Bsf}\times_\rho\ZZ)^{(2)}}\,,$$
	$$f_4=\Ind_{(Z(\beta\alpha,0,\beta\alpha;Z(\beta\alpha))\times\{0\})\times Z(\beta\alpha,0,\beta\alpha;Z(\beta\alpha))\times\{0\}\cap(\G_{\Asf,\Bsf}\times_\rho\ZZ)^{(2)}}\,,$$
that satisfy
	$$\delta_2(f_1+f_2-f_3-f_4)=\Ind_{Z(\alpha,1,\alpha;Z(\alpha))\times\{1\}}-\Ind_{Z(\beta\alpha,1,\beta\alpha;Z(\beta\alpha))\times\{0\}}\,.$$
	
	Then,
	\begin{align*}
	\tilde{\rho}^1 (\Phi^{-1}_\Bsf\left( \varphi^\Bsf_{i,\infty}(\Ind_v)\right) )& =  \hat{\rho}^1 \left(\left[  \Ind_{Z(\alpha,1,\alpha;Z(\alpha))\times\{0\}} \right]\right)  = \left[ \Ind_{Z(\alpha,1,\alpha;Z(\alpha))\times\{1\}}\right] \\
	&  = \left[ \Ind_{Z(\beta\alpha,1,\beta\alpha;Z(\beta\alpha))\times\{0\}}\right] = \Phi_\Bsf^{-1}\left( \varphi^\Bsf_{i+1,\infty}(\Ind_v)\right) \,,
	\end{align*} 
	as desired.

\end{proof}

\begin{theorem}\label{HK_conj}
Let $N\in \NN\cup\{\infty\}$, and let $\Asf$ and $\Bsf$ be a pair of pseudo free $N\times N$  row-finite matrices with integer entries, and such that $\Asf_{i,j}\geq 0$ for all $i$ and $j$.  Then 
\begin{align*}
H_0(\G_{\Asf,\Bsf})& \cong \coker (\Ide-\Asf) & \qquad & H_1(\G_{\Asf,\Bsf}) \cong \ker(\Ide-\Asf)\oplus \coker (\Ide-\Bsf) \\
H_2(\G_{\Asf,\Bsf})& \cong \ker (\Ide-\Bsf)\,, &\qquad & H_{i}(\G_{\Asf,\Bsf}) =0\text{ for }i\geq 3\,.
\end{align*}
Therefore, $\G_{\Asf,\Bsf}$ satisfies the $(HK)$ conjecture.
\end{theorem}
\begin{proof}
	By Lemma \ref{exact_seq} we have the long exact sequence 
	$$\xymatrix {  0   &  H_0(\G_{\Asf,\Bsf})\ar@{>}[l] & H_0(\G_{\Asf,\Bsf}\times_\rho\ZZ) \ar@{>}[l] & H_0(\G_{\Asf,\Bsf}\times_\rho\ZZ)\ar@{>}[l]_{\Ide-\hat{\rho}^1} & H_1(\G_{\Asf,\Bsf})\ar@{>}[l] & \cdots\ar@{>}[l] \\ \cdots   &  H_n(\G_{\Asf,\Bsf})\ar@{>}[l] & H_n(\G_{\Asf,\Bsf}\times_\rho\ZZ) \ar@{>}[l] & H_n(\G_{\Asf,\Bsf}\times_\rho\ZZ)\ar@{>}[l]_{\Ide-\hat{\rho}^1} & H_{n+1}(\G_{\Asf,\Bsf})\ar@{>}[l] & \cdots\ar@{>}[l] }\,,$$
	where $\hat{\rho}^1$ is the induced map by the action $\hat{\rho}:\ZZ\curvearrowright\G\times_\rho \ZZ$.
	Since by Lemma \ref{Morita} the groupoids $\HG_{\Asf,\Bsf}$ and $\G_{\Asf,\Bsf}\times_\rho\ZZ$  are Kakutani equivalent, then Lemma \ref{lemma_HG} says that  $H_i(\G_{\Asf,\Bsf}\times_\rho\ZZ)=0$ for $i\geq 2$. Then we have the following long exact sequence 
	$$\xymatrix {  0   &  H_0(\G_{\Asf,\Bsf})\ar@{>}[l] & H_0(\G_{\Asf,\Bsf}\times_\rho\ZZ) \ar@{>}[l] & H_0(\G_{\Asf,\Bsf}\times_\rho\ZZ)\ar@{>}[l]_{\Ide-\hat{\rho}^1} & \ar@{>}[l] \\ 
		{}   &  H_1(\G_{\Asf,\Bsf})\ar@{>}[l] & H_1(\G_{\Asf,\Bsf}\times_\rho\ZZ) \ar@{>}[l] & H_1(\G_{\Asf,\Bsf}\times_\rho\ZZ)\ar@{>}[l]_{\Ide-\hat{\rho}^1} & H_{2}(\G_{\Asf,\Bsf})\ar@{>}[l] & 0\ar@{>}[l] }\,,$$
	and $H_i(\G_{\Asf,\Bsf})=0$ for $i\geq 3$. But by Proposition \ref{trans_map} and \cite[Lemma 7.15]{Rae} we have that that $$\ker(\Ide-\hat{\rho}^1:H_0(\G_{\Asf,\Bsf}\times_\rho\ZZ)\to H_0(\G_{\Asf,\Bsf}\times_\rho\ZZ))\cong \ker (\Ide-\Asf)\,,$$ 
	$$\coker(\Ide-\hat{\rho}^1:H_0(\G_{\Asf,\Bsf}\times_\rho\ZZ)\to H_0(\G_{\Asf,\Bsf}\times_\rho\ZZ))\cong \coker (\Ide-\Asf)\,,$$ 
	$$\ker(\Ide-\hat{\rho}^1:H_1(\G_{\Asf,\Bsf}\times_\rho\ZZ)\to H_1(\G_{\Asf,\Bsf}\times_\rho\ZZ))\cong \ker (\Ide-\Bsf)\,,$$ 
	$$\coker(\Ide-\hat{\rho}^1:H_1(\G_{\Asf,\Bsf}\times_\rho\ZZ)\to H_1(\G_{\Asf,\Bsf}\times_\rho\ZZ))\cong \coker (\Ide-\Bsf)\,.$$

	Since $\ker (\Ide-\Asf)$ and $\ker (\Ide-\Bsf)$ are free abelian groups, then the exact sequence splits in the short exact sequences 
	$$0\longrightarrow\coker(\Ide-\Asf)\longrightarrow  H_0(\G_{\Asf,\Bsf})\longrightarrow  0\,,$$   
	$$0\longrightarrow \coker (\Ide-\Bsf)\longrightarrow H_1(\G_{\Asf,\Bsf})\longrightarrow \ker(\Ide-\Asf)\longrightarrow 0\,,$$
	$$0\longrightarrow H_2(\G_{\Asf,\Bsf})\longrightarrow \ker(\Ide-\Bsf)\longrightarrow 0\,,$$   
	as desired.

\end{proof}

\begin{remark} We would like to point out that the exact sequences at the end of the proof of Theorem \ref{HK_conj} are the same that one gets when using the spectral sequence described in \cite{Mat2} as this was our initial strategy. But now our proof uses a more primitive but intuitive method in homological algebra, without an extra cost in the computations.     
\end{remark}

\begin{corollary}
Let $N,N'\in \NN$, and let $\Asf\,,\Bsf\in M_N(\ZZ)$ and $\Asf'\,,\Bsf'\in M_{N'}(\ZZ)$ pseudo free pairs of matrices, such that $\Asf_{i,j},\Asf'_{i,j}\geq 0$ for all $i$ and $j$.  Suppose that $\G_{\Asf,\Bsf}$ and $\G_{\Asf',\Bsf'}$ are Kakutani equivalent. Then $\ker(\Ide-\Asf)\cong \ker(\Ide-\Asf')$ and $\ker(\Ide-\Bsf)\cong \ker(\Ide-\Bsf')$. 
\end{corollary}

\begin{example}
Let $\Asf=(2)$ and $\Bsf=(1)$, and let 
$$\Asf'=\left( \begin{array}{cc} 2 & 1 \\ 1 & 2 \end{array}\right) \,,$$
then we have that $\G_{\Asf,\Bsf}$ and $\G_{\Asf'}$ (the SFT-groupoid \cite{Mat3}) are minimal, Hausdorff, effective and purely infinite étale groupoids with compact unit space, and 
$$K_0(C^*(\G_{\Asf,\Bsf}))=K_0(C^*(\G_{\Asf'}))\cong \ZZ\qquad\text{and}\qquad K_1(C^*(\G_{\Asf,\Bsf}))=K_1(C^*(\G_{\Asf'}))\cong \ZZ\,,$$
so $\mathcal{O}_{\Asf,\Bsf}$ and $\mathcal{O}_{\Asf'}$ are stable isomorphic. But then by Theorem \ref{HK_conj} we have that
$$H_0(\G_{\Asf,\Bsf})=0\,,\qquad H_1(\G_{\Asf,\Bsf})\cong\ZZ\qquad\text{and}\qquad H_2(\G_{\Asf,\Bsf})\cong\ZZ\,,$$
while
$$H_0(\G_{\Asf'})\cong \ZZ\,,\qquad H_1(\G_{\Asf'})\cong\ZZ\qquad\text{and}\qquad H_2(\G_{\Asf'})=0\,,$$
see \cite[Theorem 4.14]{Mat2}, and therefore $\G_{\Asf,\Bsf}$ and $\G_{\Asf'}$ cannot be Kakutani equivalent.  In particular, does not exists any diagonal preserving isomorphism between the stabilizations of $\mathcal{O}_{\Asf,\Bsf}$ and $\mathcal{O}_{\Asf'}$ (see for example \cite[Theorem 3.12]{FKPS}). 

 In a private correspondence, Enrique Pardo showed me  how to prove using \cite{CEP} that the isotropy groups of $\G_{\Asf,\Bsf}$ are isomorphic either to $0$ or $\ZZ$. Therefore, homology is the invariant that distinguishes the equivalence classes of these groupoids.  
\end{example}

\section{Final remarks}

 In this final section we will use the previous computations  on the homology of the groupoid $\G_{\Asf,\Bsf}$ to give examples of groupoids with prescribed homology and satisfying the (HK) conjecture, whose associated groupoid $C^*$-algebra falls in a classifiable class.

 \begin{lemma}\label{principal}
 	Let $N\in \NN\cup\{\infty\}$, and let $\Asf$ and $\Bsf$ be two $N\times N$ row-finite matrices with integer entries, and such that $\Asf_{i,j}\geq 0$ for all $i$ and $j$, and $|\Bsf_{i,j}|<\Asf_{i,j}$ for every $(i,j)\in\Omega_\Asf$ and $\Egr$ is acyclic. Then the groupoid  $\G_{\Asf,\Bsf}$ is principal,
 \end{lemma}
 \begin{proof}
 	   Let $g=[\alpha,m,\beta;\beta x]\in\G_{\Asf,\Bsf}$ with $r(g)=s(g)=\beta x$.
 	   Since $\Egr$ is acyclic we can assume that $\alpha=\beta$. Then $r(g)=s(g)=\beta x$ if and only if $\kappa_m(x)=x$. Now by \cite[Lemma 18.4]{EP}  given $x\in\Egr^\infty$  and $m\in\ZZ$, $\kappa_m(x)=x$ if and only if $m\frac{\Bsf_{x_{|l}}}{\Asf_{x_{|l}}}\in\ZZ$ for every $l\in\NN$. But then by hypothesis it is clear that for every $x\in\Egr^\infty$  and $m\in\ZZ$ there exists $l\in\NN$ such that $m\frac{\Bsf_{x_{|l}}}{\Asf_{x_{|l}}}\notin \ZZ$.
 \end{proof}

 By a Bratteli diagram $(V,E)$, we mean a vertex set $V$, which is the union of finite non-empty sets $V_0,V_1,\ldots$, with $V_0=\{v_0\}$, and edge set $E$, which is the union of finite non-empty sets $E_1,E_2,\ldots$, such that the source and range maps restrict $s:E_n\to V_{n-1}$ and $r:E_n\to V_n$  for $n\geq 1$. In particular, a Bratteli diagram is a directed graph such that the associated incidence matrix $\Asf$ is row-finite. Moreover, if $\Asf$ is the incidence matrix of the Bratteli diagram we have that $\coker (\Ide-\Asf)$ is the $K_0$ of the associated $AF$-algebra and $\ker(\Ide-\Asf)=0$.
 
 \begin{remark}\label{dim_group}
Given a simple, acyclic dimension group $G_0$, and any dimension group $G_1$ one can find Bratteli diagrams $(V,E)$ and $(W,F)$ such that the associated $AF$-algebras have $K_0$ groups $G_0$ and $G_1$ respectively. Since $G_0$ is a simple dimension group, we can assume that between every vertex at some level  $V_n$ and any other vertex at $V_{n+1}$ there exists at least one edge. Let $\Asf$ be the adjacency matrix of  $(V,E)$ and let $\Bsf$ be the adjacency matrix of $(W,F)$. Telescoping and out-splitting $(V,E)$ we can assume that $|\Bsf_{i,j}|<\Asf_{i,j}$ for every $(i,j)\in\Omega_\Asf$ (see for example \cite[page 1368]{Put}).   
 \end{remark}
 
\begin{proposition}\label{prop_AT}
Let $G_0$ be a simple, acyclic  dimension group, and let $G_1$ be any dimension group. Then there exist $N\in\NN\cup\{\infty\}$ and $N\times N$ row-finite matrices $\Asf$ and $\Bsf$ with natural entries, such that $\G_{\Asf,\Bsf}$ is an amenable, Hausdorff,  principal,  minimal étale groupoid with
$$H_0(\G_{\Asf,\Bsf})\cong K_0(C^*(\G_{\Asf,\Bsf}))\cong  G_0\,,\qquad H_1(\G_{\Asf,\Bsf})\cong K_1(C^*(\G_{\Asf,\Bsf}))\cong G_1\,,$$
and $H_i(\G_{\Asf,\Bsf})=0$ for $i\geq 2$.
In particular, $C^*(\G_{\Asf,\Bsf})$ is a simple  $A\mathbb{T}$-algebra.
\end{proposition}
\begin{proof}
Let us consider $\Asf$ and $\Bsf$ as explained in Remark  \ref{dim_group}. Then $\G_{\Asf,\Bsf}$ is an amenable, Hausdorff and minimal groupoid groupoid \cite[Section 18]{EP}, and by Lemma \ref{principal} it is also principal.  The homology is computed in Theorem \ref{HK_conj}, so we only need to see that $C^*(\G_{\Asf,\Bsf})$ is an $A\mathbb{T}$-algebra. Let $(V,E)$ be the Bratteli diagram with incidence matrix $\Asf$, and let $V=\bigsqcup_{i\geq  0} V_i$ be the level decomposition of the diagram. Then  given $n\in\NN$ we define
$$\G_{\Asf,\Bsf,n}=\{[\alpha,n,\beta;x]\in\G_{\Asf,\Bsf}:r(\alpha)=r(\beta)\in V_n\}\,,$$
with the subspace topology. It is an open subgroupoid of $\G_{\Asf,\Bsf}$ and we have that $\G_{\Asf,\Bsf}=\bigcup_{n=0}^\infty \G_{\Asf,\Bsf,n}$. Given $v\in\Egr^0$, let $\un_v$ be the partial unitary $\Ind_{Z(v,1,v)}\in C^*(\G_{\Asf,\Bsf})$. Then we have that $C^*(\G_{\Asf,\Bsf,n})\cong \bigoplus_{v\in V_n}M_{n_v}(C(\text{spec}(\un_v)))\otimes C(Z(v))$ where $n_v=\sharp\{\alpha\in \Egr:r(\alpha)=v\}$, which is an $A\mathbb{T}$-algebra. Then by \cite[Proposition 1.9]{Phil} we have that $C^*(\G_{\Asf,\Bsf,n})$ is a subalgebra of $C^*(\G_{\Asf,\Bsf})$, and hence $C^*(\G_{\Asf,\Bsf})=\bar{\bigcup_{n\geq 0}C^*(\G_{\Asf,\Bsf,n})}$, whence $C^*(\G_{\Asf,\Bsf})$ is an $A\mathbb{T}$-algebra.

\end{proof}

\begin{remark}
The groupoids in Proposition \ref{prop_AT} and \cite{Put} look very similar in the way they are constructed. However, the author does not know  whether $K_0(C^*(\G_{\Asf,\Bsf}))$ and $G_0$ are isomorphic as ordered groups.  The map $\lambda:H_0(\G_{\Asf,\Bsf})\to K_0(C^*(\G_{\Asf,\Bsf}))$ given by $[\Ind_{Z(v)}]\to [p_v]$ for every $v\in\Egr^0$, is a group isomorphism (\cite[Proposition 2.6]{Kat}), but is unclear if it is an isomorphism of ordered groups. 
\end{remark} 
 

For the rest of the section we will assume that $\Asf$ and $\Bsf$ are the incidence matrices of two Bratteli diagrams $(V,E)$ and $(W,F)$ respectively, defined in Remark \ref{dim_group}, satisfying that  $|\Bsf_{i,j}|<\Asf_{i,j}$ for every $(i,j)\in\Omega_\Asf$. 
 
In general, the unit space of the groupoid $\G_{\Asf,\Bsf}$ is not compact, and hence $C^*(\G_{\Asf,\Bsf})$ is not a unital $C^*$-algebra. We can define the groupoid $\tilde{\G}_{\Asf,\Bsf}:=(\G_{\Asf,\Bsf})_{Z(v_0)}$ where $v_0$ is the initial vertex of the Bratteli diagram $(V,E)$. Then the groupoid $\tilde{\G}_{\Asf,\Bsf}$ is  amenable, Hausdorff, principal, minimal and étale, and has a compact unit space homeomorphic to $Z(v_0)$. Moreover, since $Z(v_0)\subseteq \G_{\Asf,\Bsf}^{(0)}$ is  $\G_{\Asf,\Bsf}$-full, we have that $\G_{\Asf,\Bsf}$ and $\tilde{\G}_{\Asf,\Bsf}$ are Kakutani equivalent, whence $H_i(\G_{\Asf,\Bsf})\cong H_i(\tilde{\G}_{\Asf,\Bsf})$ for $i\geq 0$, and $C^*(\tilde{\G}_{\Asf,\Bsf})$ is a unital $A\mathbb{T}$-algebra Morita equivalent to $C^*(\G_{\Asf,\Bsf})$.

Given a groupoid $\G$ with compact unit space $\G^{(0)}$, we denote by $M(\G)$ the set of probability measures $\mu$ of $\G^{(0)}$ such that given any bisection $U\subseteq G$ we have that $\mu(s(U))=\mu(r(U))$.

 \begin{lemma}
	Let $\Asf$ and $\Bsf$ be the incidence matrices of two Bratteli diagrams $(V,E)$ and $(W,F)$ respectively, satisfying that  $|\Bsf_{i,j}|<\Asf_{i,j}$ for every $(i,j)\in\Omega_\Asf$. Then $M(\tilde{\G}_{\Asf,\Bsf})=M(\tilde{\G}_{\Asf,0})$.
\end{lemma}
\begin{proof}
Clearly $M(\tilde{\G}_{\Asf,\Bsf})\supseteq M(\tilde{\G}_{\Asf,0})$. On the other hand given $\eta\in M(\tilde{\G}_{\Asf,0})$, $\alpha\in \Egr^*$ with $s(\alpha)=v_0$, and $m\in\ZZ$, we have that $\eta(Z(\alpha))=\eta(Z(\kappa_m(\alpha)))$ because the bisection $U=Z(\kappa_m(\alpha),0,\alpha;Z(\alpha))\subseteq \tilde{\G}_{\Asf,0}$ is such $s(U)=Z(\alpha)$ and $r(U)=Z(\kappa_m(\alpha))$. 
\end{proof}

 Every $\mu\in M(\G)$ induces a trace  $\mu\circ E$ on $C_r^*(\G)$ (viewing $\mu$ as a state of $C(\G^{(0)})$),  where $E:C_r^*(\G)\to C(\G^{(0)})$ is the canonical conditional expectation.  Moreover, if $\G$ is a principal groupoid every trace $\tau$ of $C_r^*(\G)$ satisfies $\tau\circ E=\tau$ (see \cite[Lemma 4.3]{LiRe} for example). Observe that given two different $\mu_1,\mu_2\in M(\G)$ induce two different traces $\mu_1\circ E$ and $\mu_2\circ E$ of $C^*_r(\G)$. Therefore, we have bijection between $M(\G)$ and $T(C^*_r(\G))$, the traces of $C^*_r(\G)$. 
 
 Then given row-finite matrices $\Asf$ and $\Bsf$, there is a bijection between $T(C^*(\tilde{\G}_{\Asf,0}))$ and $M(\tilde{\G}_{\Asf,\Bsf})$.  But $C^*(\tilde{\G}_{\Asf,0})$ is a simple unital $AF$-algebra, and hence by \cite{Bla} and Proposition \ref{prop_AT} for each metrizable Choquet simplex $\Delta$ there exists $\Asf$ such that $T(C^*(\tilde{\G}_{\Asf,0}))$ is homeomorphic to $\Delta$.

 Finally, we present a last example of a minimal, purely infinite étale groupoid with a prescribed homology. The example covers partially the result of Li and Renault \cite[Lemma 5.5]{LiRe}.
 
 \begin{proposition}
 Let $G_0$ be a simple, acyclic dimension group, and let $G_1$ be any dimension group. Then there exist an amenable, Hausdorff,  effective, purely infinite, minimal étale  groupoid  $\G$ with unit space homeomorphic to the Cantor space and isotropy groups isomorphic either to $0$ or to $\ZZ$, that satisfies the (HK) conjecture, and with 
 $$K_0(C^*(\G))\cong H_0(\G)\cong G_0\qquad\text{and}\qquad K_1(C^*(\G))\cong H_1(\G)\cong G_1\,.$$
 \end{proposition}
\begin{proof}
Let $\Asf$ and $\Bsf$ be from Proposition \ref{prop_AT}, and let $\tilde{\G}_{\Asf,\Bsf}$, that is a principal étale groupoid, with  $\tilde{\G}_{\Asf,\Bsf}^{(0)}$ homeomorphic to the Cantor space, and  with $H_0(\tilde{\G}_{\Asf,\Bsf})\cong G_0$ and $H_1(\tilde{\G}_{\Asf,\Bsf})\cong G_1$. Now let $\G_\infty$ be any graph groupoid such that $C^*(\G_\infty)\cong \mathcal{O}_\infty$, that is an amenable, Hausdorff, minimal, effective and purely infinite étale groupoid, with $\G_\infty^{(0)}$ homeomorphic to the Cantor space and isotropy groups isomorphic to either $0$ or $\ZZ$. It is computed in \cite{NO} that $H_0(\G_\infty)\cong \ZZ$ and $H_i(\G_\infty)=0$ for $i\geq 1$. Then the groupoid $\G:=\tilde{\G}_{\Asf,\Bsf}\times \G_\infty$ is  an amenable, Hausdorff, minimal, effective and purely infinite, with $\G^{(0)}$ homeomorphic to the Cantor space and isotropy groups isomorphic to either $0$ or $\ZZ$, and  by \cite[Theorem 2.4 \& Theorem 2.8]{Mat4} the rest of the statement follows.  
\end{proof}

The groupoids constructed in the above Proposition have much simple isotropy groups than the general groupoids $\G_{\Asf,\Bsf}$ \cite{CEP}

\section*{Acknowledgments}
The author is grateful to Enrique Pardo for sharing private notes and fruitful conversations, and to Hiroki Matui for valuable feed-back from the first draft of the paper. Finally, the author wish to thank the referee for many useful suggestions that have
improved the paper.

\end{document}